\documentclass[a4paper, article,  11pt, oneside]{memoir}

\counterwithout{section}{chapter}
\usepackage[T1]{fontenc}

\usepackage[scaled=0.85]{berasans}
\usepackage[scaled=0.84]{beramono}

\usepackage{titletoc}
\usepackage{titlesec}
\titleformat{\chapter}[display]
  {\normalfont\sffamily\huge\bfseries}
  {\chaptertitlename\ \thechapter}{20pt}{\Huge}
\titleformat{\section}
  {\normalfont\sffamily\Large\bfseries}
  {\thesection}{1em}{}
	\titleformat{\subsection}
  {\normalfont\sffamily\large\bfseries}
  {\thesubsection}{1em}{}

\usepackage{amsmath,amssymb, amsfonts, mathrsfs}
 \numberwithin{equation}{section}
\usepackage[amsmath,thmmarks]{ntheorem}

\usepackage[euler-hat-accent, euler-digits]{eulervm}
\usepackage{euscript}
\usepackage{eufrak}
\usepackage[english]{babel}
\usepackage[utf8x]{inputenc}
\usepackage{bbm}

\usepackage{enumitem}
\usepackage{nameref}
\usepackage{cite}
\usepackage{lastpage, setspace}

\usepackage{graphicx, pst-all}
\usepackage[all]{xy}

\usepackage{tikz}
\usetikzlibrary{matrix}
\usepackage{multicol}
\usepackage{listings}
\usepackage{verbatim}
\usepackage{pdflscape}
\usepackage{color}
\usepackage{bm}

\usepackage{setspace}

\usepackage{blindtext}

\makeatletter
\newcommand{\displaybump}{\hbox to \@totalleftmargin{\hfil}}
\makeatother

\psset{linewidth=0.4pt}
\psset{arrowsize=3pt}
\SelectTips{cm}{}
\newdir{ >}{{}*!/-5pt/\dir{>}}

\DeclareMathOperator\met{(X,d)}

\theoremstyle{plain}
\theoremseparator{.}
\newtheorem{theorem}{\sffamily{Theorem}}[section]

\newtheorem{corollary}[theorem]{\sffamily{Corollary}}

\newtheorem{lemma}[theorem]{\sffamily{Lemma}}
\newtheorem{proposition}[theorem]{\sffamily{Proposition}}

\theoremsymbol{\ensuremath{\clubsuit}}

\theoremstyle{nonumberplain}
\theorembodyfont{\normalfont}
\theoremsymbol{\ensuremath{\blacksquare}}
\newtheorem{proof}{\sffamily{Proof}}
\newtheorem{proff}{\sffamily{Proof of Theorem \ref{maintheorem}}}
\newtheorem{profff}{\sffamily{Proof of Theorem \ref{compactfixedpoint}}}
\newcommand{\Q}{\mathbb{Q}}
\newcommand{\R}{\mathbb{R}}
\newcommand{\Z}{\mathbb{Z}}

\newcommand\norm[1]{\left\lVert#1\right\rVert}

\newcommand\abs[1]{\left\lvert#1\right\rvert}
\renewcommand{\epsilon}{\ensuremath\varepsilon}

\renewcommand{\phi}{\ensuremath{\varphi}}

\newcommand*{\TitleFont}{%
      \usefont{\encodingdefault}{\sfdefault}{b}{n}%
      \fontsize{18}{20}%
      \selectfont}

\title{\TitleFont Fixed point theorems for metric spaces with a conical geodesic bicombing}
\author{Giuliano Basso}
\date{\today}

\begin{document}

\maketitle

%\begin{abstract}
%Short description of content here. Lorem ipsum dolor sit amet, consectetuer adipiscing elit. Aenean commodo ligula eget dolor. Aenean massa. Cum sociis natoque penatibus et magnis dis parturient montes, nascetur ridiculus mus. Donec quam felis, ultricies nec, pellentesque eu, pretium quis, sem. Nulla consequat massa quis enim. Donec pede justo, fringilla vel, aliquet nec, vulputate eget, arcu. In enim justo, rhoncus ut, imperdiet a, venenatis vitae, justo. 
%\end{abstract}
%\cleartorecto
\renewcommand{\abstractname}{\sffamily{Abstract}}
\begin{abstract}
\noindent We derive two fixed point theorems for a class of metric spaces that includes all Banach spaces and all complete Busemann spaces. We obtain our results by the use of a \(1\)-Lipschitz barycenter construction and an existence result for invariant Radon probability measures. Furthermore, we construct a bounded complete Busemann space that admits an isometry without fixed points. 
\end{abstract}
%\renewcommand{\contentsname}{\sffamily{Contents}}
%{\sffamily\tableofcontents*}
%-----------------------------------------------------------------------------------------------------------------------------------------------------------------------------------------------------------
%-----------------------------------------------------------------------------------------------------------------------------------------------------------------------------------------------------------

\section{Introduction}
Let \((X,d)\) be a metric space. It is well-known that if \((X,d)\) is a complete CAT(0) space, then every subgroup \(\Gamma\) of the isometry group of \(X\) with bounded orbits has a non-empty fixed point set, cf. \cite[Corollary II.2.8]{bridson}. Analogous results are known for a wide variety of metric spaces. For example, the above statement holds if the metric space \((X,d)\) is an L-embedded Banach space or an injective metric space, cf. \cite[Theorem A]{badermonod2012fixed} and \cite[Proposition 1.2]{lang2013injective}. Further results can be found in \cite{leust, edelstein}.

A metric space \(\met\) is said to be a \textit{Busemann space} if \((X,d)\) is a geodesic metric space and if the function \(t\mapsto d(\sigma(t), \tau(t))\) is convex on \([0,1]\) for all pairs of constant speed geodesics \(\sigma, \tau\colon [0,1]\to X\). It turns out that our initial statement does not hold if \((X,d)\) is a complete Busemann space instead of a complete CAT(0) space. A counterexample is discussed in Section \ref{three}.

The first main result of this article, Theorem \ref{compactfixedpoint}, is a fixed point theorem for the class of metric spaces that admit a conical geodesic bicombing. We follow D. Descombes and U. Lang and say that a map \(\sigma\colon X\times X\times [0,1]\to X\) is a \textit{conical geodesic bicombing} if for all points \(x,y\) in \(X\) the map \(\sigma_{xy}(\cdot):=\sigma(x,y,\cdot)\) is a constant speed geodesic from \(x\) to \(y\) (meaning that \(\sigma_{xy}(0)=x\), \(\sigma_{xy}(1)=y\) and \(d(\sigma_{xy}(s), \sigma_{xy}(t))=\abs{s-t}d(x,y)\) for all \(s,t\) in \([0,1]\)) and if the inequality
\begin{equation*}
d(\sigma_{xy}(t), \sigma_{x^\prime y^\prime}(t))\leq (1-t)d(x,x^\prime)+td(y,y^\prime)
\end{equation*}
holds for all points \(x,x^\prime,y,y^\prime\) in \(X\) and all \(t\) in \([0,1]\), cf. \cite{lang1}. The term bicombing was coined by D. Epstein and W. Thurston in the context of combinatorial group theory, cf. \cite[p. 84]{Epstein:1992:WPG:573874}. To the author's knowledge, conical geodesic bicombings have first been considered by S. Itoh under the name \(W\)-convexity mappings that satisfy condition (III), cf. \cite{itoh1979some}. Examples of metric spaces that admit a conical geodesic bicombing include all convex subsets of normed vector spaces and all Busemann spaces. Furthermore, every injective metric space admits a conical geodesic bicombing, cf. \cite[p. 369]{lang1}. For further examples and a thorough discussion of conical geodesic bicombings we refer to \cite{lang1}. 

%Conical geodesic bicombings have also been considered in \cite{leust, itoh1979some}.

We say that a conical geodesic bicombing \(\sigma\colon X\times X\times [0,1]\to X\) has the \textit{midpoint property} if \(\sigma_{xy}(\frac{1}{2})=\sigma_{yx}(\frac{1}{2})\) for all points \(x,y\) in \(X\). Let \(A\subset X\) be a subset of \(X\). The set \(\textrm{conv}_{\sigma}(A):= \bigcup_{k\geq 1} A_k\), where the sequence \(\big(A_k\big)_{k\geq 1}\) of subsets of \(X\) is given by the recursive rule
\begin{equation*}
A_1:= A, \quad A_{k+1}:=\big\{ \sigma_{xy}(t) : x,y\in A_{k}, t\in [0,1]\big\},
\end{equation*}
is called the \textit{\(\sigma\)-convex hull of \(A\)}. 
%Note that if (X, d) is a CAT(0) or a Busemann space and assume that \(\sigma\) is the conical geodesic bicombing given by the unique geodesics. It is immediate that in this case the \(\sigma\)-convex hull coincides with the standard convex hull.
We use \(\overline{\textrm{conv}_{\sigma}}(A)\) to denote the closure of the \(\sigma\)-convex hull of \(A\). Let \(\phi\colon X\to X\) be an isometry of \((X,d)\) and let \(\sigma\colon X\times X\times [0,1]\to X\) be a conical geodesic bicombing. We say that \(\phi\) is \textit{\(\sigma\)-equivariant} if \(\phi\circ\sigma_{xy}=\sigma_{\phi(x)\phi(y)}\) for all points \(x,y\) in \(X\). Moreover, a subsemigroup \(\Sigma\) of the isometry group of \(X\) is called \(\sigma\)-equivariant if every isometry of \(\Sigma\) is \(\sigma\)-equivariant.
Now, we have everything on hand to state our first result.
\begin{theorem}\label{compactfixedpoint}
Let \((X,d)\) denote a complete metric space and let \(\sigma\colon X\times X\times [0,1]\to X\) be a conical geodesic bicombing that has the midpoint property. Suppose that \(
\Sigma\) is a \(\sigma\)-equivariant subsemigroup of the isometry group of \(X\). If there is a non-empty compact subset \(K\subset X\) such that \(s(K)=K\) for all \(s\) in \(\Sigma\), then there is a point \(x_\star\) in \(\overline{\textrm{conv}_{\sigma}}(K)\) such that \(s(x_\star)=x_\star\) for all \(s\) in \(\Sigma\). 
\end{theorem}

In \cite[p. 620]{navas}, A. Navas introduced a simple geometric argument that implies Theorem \ref{compactfixedpoint} if one requires additionally that the closed convex hull of \(K\) is compact. Unfortunately, Navas's method seems not to work without this additional assumption. 

The proof strategy of Theorem \ref{compactfixedpoint} may be roughly described as follows: We use Ryll-Nardzewski's fixed point theorem to construct an invariant Radon probability measure first, and then we use the equivariant contracting barycenter map from Theorem \ref{navas1} to obtain a fixed point.

%If the complete metric space \((X,d)\) is a Banach space, then Theorem \ref{compactfixedpoint} is a direct consequence of the Ryll-Nardzewski fixed point theorem, cf. \cite{ryll1967fixed}. In fact, our proof of Theorem \ref{compactfixedpoint} relies on the Ryll-Nardzewski fixed point theorem. 

How restrictive is the assumption in Theorem \ref{compactfixedpoint} that the subsemigroup \(\Sigma\) is \(\sigma\)-equivariant? Fortunately, metric spaces \((X,d)\) that admit a conical geodesic bicombing \(\sigma\colon X\times X\times [0,1]\to  X\) such that every isometry of \(X\) is \(\sigma\)-equivariant provide a broad spectrum of examples. For instance, every isometry of a Busemann space is \(\sigma\)-equivariant with respect to the conical geodesic bicombing \(\sigma\) given by the unique geodesics. Moreover, it is a consequence of a generalised version of the Mazur-Ulam Theorem that every isometry of an open convex subset of a normed vector space is \(\sigma\)-equivariant with respect to the conical geodesic bicombing \(\sigma\) given by the linear geodesics, cf. \cite[p. 368]{mankiewi1972extension}. Furthermore, Proposition 3.8 in \cite{lang2013injective} asserts that every injective metric space \((X,d)\) admits a conical geodesic bicombing \(\sigma\) such that every isometry of \(X\) is \(\sigma\)-equivariant.

%Note that if the subsemigroup \(\Sigma\) from Theorem \ref{compactfixedpoint} is a group, then the condition that \(s(K)=K\) holds for all \(s\) in \(\Sigma\) is equivalent to the condition that \(\Sigma.K=K\). 

Our second main result is a strengthened version of Theorem \ref{compactfixedpoint} if the subsemigroup \(\Sigma\) is generated by a single isometry. 

%Let \(\phi\colon X\to X\) be an isometry. We denote by \(\Sigma(\phi)\)\) the subsemigroup of the isometry group of \(X\) generated by the isometry \(\phi\). 

\begin{theorem}\label{maintheorem} Let \((X,d)\) denote a complete metric space and let \(\sigma\colon X\times X\times [0,1]\to X\) be a conical geodesic bicombing that has the midpoint property. Suppose that \(\phi\colon X\to X\) is a \(\sigma\)-equivariant isometry of \((X,d)\). If there is a point \(x_0\) in \(X\) and a compact subset \(K\subset X\) such that the strict inequality 
\begin{equation}\label{strictIneq}
\limsup_{k\to +\infty}\left(\sup_{l\geq 0} \frac{1}{k} \sum_{i=0}^{k-1} \mathbbm{1}_{K}(\phi^{i+l}(x_0))\right)> 0
\end{equation}
holds, then there is a point \(x_\star\) in \(\overline{\textrm{conv}_{\sigma}}\left(\big\{ \phi^k(x_0) : k\geq 0\big\}\right)\) such that \(\phi(x_\star)=x_\star\). 
\end{theorem}
The function \(\mathbbm{1}_{K}\colon X\to \{0,1\}\) in Theorem \ref{maintheorem} denotes the indicator function of the subset \(K\subset X\).

%Recall that the sequence \((\{0, \ldots, k-1\})_{k\geq 1}\) is a Følner sequence of the semigroup of the non-negative integers (the definition of a Følner sequence is given in Section \ref{maxval}). Hence, the left hand side of \eqref{strictIneq} can be written as
%\begin{equation*}
%\limsup\limits_{k\to+\infty}\,\left(\sup\limits_{s\in \Sigma} \frac{1}{\lvert\Sigma_k \rvert} \sum\limits_{h\in \Sigma_k} x_{hs}\right),
%\end{equation*}
%where \((\Sigma, \circ)\) is a semigroup, \((\Sigma_k)_{k\geq 1}\) is a Følner sequence and \((x_s)_{s\in \Sigma}\) is a point of \(\ell^\infty(\Sigma)\). 

Note that the left hand side of \eqref{strictIneq} is equal to the \textit{upper Banach density}, cf. \cite[Definition 3.7]{furstenberg2014recurrence}, of the set \(D:=\left\{ k\geq 0 : \mathbbm{1}_{K_0}(\phi^{k}(x_0))=1 \right\}\). This fact allows us to invoke a basic result from combinatorial number theory in order to show that the orbits of the isometry \(\phi\) are bounded, see Lemma \ref{bounded}.

One key ingredient in the proof of Theorem \ref{maintheorem} is a generalization of a classical existence result for invariant Radon measures, see Theorem \ref{mainTheorem}.
%The proof strategies of our main results, Theorem \ref{compactfixedpoint} and Theorem \ref{maintheorem}, may be roughly described as follows: In both proofs, we construct an invariant Radon probability measure first, and then we use the equivariant barycenter map from Theorem \ref{navas1} in order to obtain a fixed point. 

%In the proof of Theorem \ref{compactfixedpoint} we use the Ryll-Nardzewski fixed point theorem to construct an invariant Radon probability measure and in the the proof of Theorem \ref{maintheorem} we use a generalisation of a Theorem of J. Oxtoby and S. Ulam, see Theorem \ref{mainTheorem}. 

The paper is structured as follows: In Section \ref{three}, we construct a bounded complete Busemann space that admits an isometry without fixed points. In Section \ref{six}, we transfer a \(1\)-Lipschitz barycenter construction introduced by A. Es-Sahib and H. Heinich, cf. \cite{heinich}, and A. Navas, cf. \cite{navas}, into the setting of metric spaces that admit a conical geodesic bicombing. The primary result of Section \ref{six} is Theorem \ref{navas1}. 
%As a byproduct of our barycenter construction we obtain Corollary \ref{consequence}.
%states that for every Radon probability measure \(\mu\) on a Banach space \((E, \norm{\cdot})\) we have that if \(\mu\) has a finite first moment, then the Bochner integral \(\int_E x\,d\mu(x)\) is contained in the closure of the convex hull of \(\textrm{spt}(\mu)\). 
The purpose of Section \ref{seven} is to derive Theorem \ref{mainTheorem}, which is a generalization of a classical result from topological dynamics that is due to J. Oxtoby and S.  Ulam, cf. \cite[Theorem 1]{ulamOxtoby1939}. The results from Section \ref{seven} may be of independent interest. Finally, in Section \ref{five} we use Theorem \ref{navas1} and Theorem \ref{mainTheorem} to prove our main results. 
\section{Counterexample}\label{three}
In this section, we construct a bounded complete Busemann space that admits an isometry without fixed points. As usual, let \(\ell^1(\Z)\subset \R^\Z\) denote the linear subspace of \(\R^\Z\) that consists of all sequences \(x:=(x_k)_{k\in \Z}\) such that
\begin{equation*}
\norm{x}_1:= \sum_{k\in \Z} \abs{x_k} < +\infty. 
\end{equation*}
Now, we use a standard technique, cf. \cite[p. 786]{johnson2001}, to renorm the Banach space \((\ell^1(\Z), \norm{\cdot}_1)\) into a strictly convex Banach space. We define the map \(\norm{\cdot}_{\star}\colon \ell^1(\Z) \to \R\) through the assignment
\begin{equation*}
(x_k)_{k\in \Z} \mapsto \sqrt{\left(\sum_{k\in \Z}\abs{x_k}\right)^2 +\sum_{k\in \Z}\abs{x_k}^2}.
\end{equation*}
It is straightforward to show that the map \(\norm{\cdot}_\star\) defines a norm on \(\ell^1(\Z)\). Elementary estimates show that
\begin{equation*}%\label{CompletenessForStarNorm}
\frac{1}{\sqrt{2}}\norm{\cdot}_\star \leq \norm{\cdot}_1 \leq \norm{\cdot}_\star;
\end{equation*}
hence, the norms \(\norm{\cdot}_1\) and \(\norm{\cdot}_\star\) are equivalent. It follows that \((\ell^1(\Z), \norm{\cdot}_\star)\) is a Banach space. Recall that a normed vector space \((V, \norm{\cdot}_V)\) is said to be \emph{strictly convex} if for all distinct points \(x,y\) in \(V\) with \(\norm{x}_V=\norm{y}_V=1\) and for all \(\lambda\) in \((0,1)\) we have the strict inequality \(\norm{(1-\lambda)x+\lambda y}_V < 1 \). 
\begin{lemma}
The Banach space \((\ell^1(\Z), \norm{\cdot}_\star)\) is strictly convex.
\end{lemma}
\begin{proof}
Let \(x\) and \(y\) denote two distinct points of \(\ell^1(\Z)\) that satisfy \(\norm{x}_\star=\norm{y}_\star=1\) and let \(\lambda\) in \((0,1)\) be a real number. Since \(x\) and \(y\) are distinct, there is an integer \(k_0\) such that \(x_{k_0}\neq y_{k_0}\). It follows that 
\begin{equation}\label{strictconvexity}
\big((1-\lambda)x_{k_0}+\lambda y_{k_0}\big)^2 < (1-\lambda)x_{k_0}^2+\lambda y_{k_0}^2,
\end{equation}
as the real valued function \(f\colon \R\to \R\) given by \(f(x)=x^2\) is strictly convex. Now, elementary estimates and the strict inequality in \eqref{strictconvexity} imply that \(\norm{(1-\lambda)x+\lambda y}_{\star}^2 < (1-\lambda)\norm{x}_\star^2+\lambda \norm{y}_\star^2=1\); hence, the Banach space \((\ell^1(\Z), \norm{\cdot}_\star)\) is strictly convex, as was to be shown. 
\end{proof}
The shift map \(T\colon \ell^1(\Z)\to \ell^1(\Z)\) given by the assignment
\begin{equation*}
(x_k)_{k\in \Z}\mapsto (x_{k-1})_{k\in \Z}
\end{equation*}
is a linear map and an isometry of \((X, \norm{\cdot}_\star)\). Note that the zero sequence is the only fixed point of \(T\). Let \(x_0\in \ell^1(\Z)\) be the sequence that is equal to one if \(k=0\) and  equal to zero if \(k\neq 0\). We define the set \(A:=\big\{ T^k(x_0) : k\in \Z\big\}\). Let \(\textrm{conv}(A)\) denote the convex hull of \(A\). 
\begin{lemma}\label{LowerBound}
If \(x\) is an element of \(\textrm{conv}(A)\), then we have \(1 \leq \norm{x}_\star \leq \sqrt{2}\). 
\end{lemma}
\begin{proof}
Let \(x\) be an element of \(\textrm{conv}(A)\). By the definition of \(\textrm{conv}(A)\), there is an integer \(n\geq 0\), an element \((\alpha_0, \ldots, \alpha_n)\) in the \(n\)-dimensional standard simplex \(\Delta^n\subset \R^{n+1}\) and \(n+1\) distinct integers \(l_0, \ldots, l_n\) such that \(x=\alpha_0T^{l_0}(x_0)+\dotsm +\alpha_nT^{l_n}(x_0)\). We have for every integer \(0 \leq i \leq n\) that the sequence \(T^{l_i}(x_0)\in \ell^1(\Z)\) is equal to one if \(k=l_i\) and  equal to zero if \(k\neq l_i\). Therefore, we compute
\begin{equation*}
\begin{split}
\left(\sum_{k\in \Z} \abs{x_k}\right)^2 
=\left(\sum_{i=0}^n \abs{\alpha_i} \right)^2
=1
\end{split}
\end{equation*}
and
\begin{equation*}
\begin{split}
\sum_{k\in \Z}\abs{x_k}^2
=\sum_{i=0}^n \alpha_i^2.
\end{split}
\end{equation*}
As a result, we obtain that 
\begin{equation*}
\begin{split}
1\leq\norm{x}_\star=\sqrt{1+ \sum_{i=0}^n \alpha_i^2} \leq \sqrt{2},
\end{split}
\end{equation*}
since we have \( \alpha_i^2 \leq \alpha_i \) for all integers \(0 \leq i \leq n\). 
\end{proof}
Let \(\overline{\textrm{conv}}(A)\) denote the closure of \(\textrm{conv}(A)\). By the use of Lemma \ref{LowerBound} we obtain that \(1 \leq \norm{x}_\star \leq \sqrt{2}\) for all points \(x\) in \(\overline{\textrm{conv}}(A)\). Thus, we have in particular that the zero sequence is not an element of \(\overline{\textrm{conv}}(A)\). A straightforward calculation shows that \(T(\overline{\textrm{conv}}(A))=\overline{\textrm{conv}}(A)\); hence, the map \(T\) is an isometry of the bounded metric space \((\overline{\textrm{conv}}(A), \norm{\cdot}_{\star})\) without fixed points. We claim that \((\overline{\textrm{conv}}(A),  \norm{\cdot}_{\star})\) is a complete Busemann space. It is well-known that every convex subset of a strictly convex normed vector space is a Busemann space, cf. Proposition 8.1.6 and Proposition 8.1.5 in \cite{PapadopoulosAthanase}. Hence, it follows that \((\overline{\textrm{conv}}(A), \norm{\cdot}_{\star})\) is a Busemann space, as \(\overline{\textrm{conv}}(A)\) is a convex subset of \(\ell^1(\Z)\). Note that \((\overline{\textrm{conv}}(A), \norm{\cdot}_{\star})\) is complete. Thus, we have constructed a complete bounded Busemann space that admits an isometry without fixed points.

\section{Existence of contracting barycenter maps}\label{six}
This section is divided into two parts. In the first part we build up prerequisite material from optimal transportation theory and in the second part we transfer a \(1\)-Lipschitz barycenter construction that traces back to A. Es-Sahib and H. Heinich, cf. \cite{heinich}, and A. Navas, cf. \cite{navas}, into the setting of metric spaces that admit a conical geodesic bicombing. The primary result of this section is Theorem \ref{navas1}. 
\subsection{Prerequisite material from optimal transportation theory}\label{introductionofnotions}
In this subsection, we collect some facts from the theory of optimal transportation. Let \((\mathfrak{X}, \mathcal{T}_{\hspace{-0.05em}\mathfrak{X}})\) be a Hausdorff topological space. We denote by \(\mathcal{B}_{\hspace{-0.05em}\mathfrak{X}}\) the Borel \(\sigma\)-algebra of \((\mathfrak{X}, \mathcal{T}_{\hspace{-0.05em}\mathfrak{X}})\) and by \(\mathcal{K}_{\hspace{-0.05em}\mathfrak{X}}\) the set that consists of all compact subsets of \((\mathfrak{X}, \mathcal{T}_{\hspace{-0.05em}\mathfrak{X}})\). A non-negative Borel measure \(\mu\colon \mathcal{B}_{\hspace{-0.05em}\mathfrak{X}}\to [0,+\infty]\) is called a \textit{Radon measure} if \(\mu(K)<+\infty\) for all compact subsets \(K\) of \((\mathfrak{X}, \mathcal{T}_{\hspace{-0.05em}\mathfrak{X}})\) and 
\begin{equation*}
\mu(B)=\sup\big\{ \mu(K) : K\subset B, K\in\mathcal{K}_{\hspace{-0.05em}\mathfrak{X}}\big\}
\end{equation*}
for all Borel measurable sets \(B\) of \((\mathfrak{X}, \mathcal{T}_{\hspace{-0.05em}\mathfrak{X}})\). A signed finite Borel measure \(\mu\colon \mathcal{B}_{\hspace{-0.05em}\mathfrak{X}}\to \R\) is called a signed finite Radon measure if the total variation \(\abs{\mu}\colon \mathcal{B}_{\hspace{-0.05em}\mathfrak{X}}\to [0,+\infty)\) is a Radon measure. Let \(P(\mathfrak{X})\) denote the set that consists of all non-negative Borel measures on \((\mathfrak{X}, \mathcal{B}_{\hspace{-0.05em}\mathfrak{X}})\) that are Radon probability measures. Suppose that \(f\colon \mathfrak{X}\to\mathfrak{X}\) is a continuous map. We define the map
\begin{equation*}
\begin{split}
&f_\ast\colon P(\mathfrak{X})\to P(\mathfrak{X}) \\
&\mu\mapsto 
\begin{cases}
f_\ast\mu\colon \mathcal{B}_{\hspace{-0.05em}\mathfrak{X}}\to [0,1] \\
B\mapsto \mu(f^{-1}(B)).
\end{cases}
\end{split}
\end{equation*}
The map \(f_\ast\) is well-defined and for every \(\mu\) in \(P(\mathfrak{X})\) the measure \(f_\ast\mu\) is called the \textit{pushforward of \(\mu\) by \(f\)}. A measure \(\mu\) in \(P(\mathfrak{X})\) is called \textit{\(f\)-invariant} if \(f_\ast\mu=\mu\).

For the rest of this subsection let \((X,d)\) be a metric space. Suppose that the Borel measure \(\mu\colon \mathcal{B}_X\to [0,1]\) is a Radon probability measure. The subset \(\textrm{spt}(\mu)\) of all points \(x\) in \(X\) such that \(\mu(U)>0\) for all open neighborhoods of \(x\) is called the \textit{support} of \(\mu\). We say that \(\mu\) has a \textit{finite first moment} if there is a point \(x_0\) in \(X\) such that 
\begin{equation*}
\int_X d(x,x_0)\,\mu(dx) < +\infty. 
\end{equation*}
We let \(P_1(X)\) be the set that consists of all measures of \((X, \mathcal{B}_X)\) that are Radon probability measures with a finite first moment. We denote by \(W^1\colon P_1(X)\times P_1(X)\to \R\) the first Wasserstein distance on \(P_1(X)\). Due to the Kantorovich-Rubinstein Duality Theorem the first Wasserstein distance \(W^1\) is given by
\begin{equation*}
W^1(\mu, \nu)=\sup\left\{ \int_X f\,d\mu-\int_X f\,d\nu \, : \, f\colon X\to \R \textrm{ is 1-Lipschitz } \right\}
\end{equation*}
and thus defines a metric on \(P_1(X)\), cf. \cite{edwards}. We define
\begin{equation*}
P_{\Q}(X):= \left\{ \sum_{k=1}^n \alpha_k \delta_{x_k} \, : \, n\geq 1, \sum_{k=1}^n \alpha_k=1, \alpha_k\in \Q_{\geq 0}, x_k\in X\right\}.
\end{equation*}
The set \(P_\Q(X)\) is  called  the  \textit{set  of  atomic  probability  measures with non-negative rational masses}. 
On \(P_\Q(X)\) there is an explicit formula for the first Wasserstein distance.
\begin{proposition}\label{formula}
Let \((X,d)\) denote a metric space. If \(n\geq 1\) is an integer and \(x_1, y_1,\ldots, x_n, y_n\) are points in \(X\), then we have
\begin{equation*}
W^1\left(\frac{1}{n}\left( \delta_{x_1}+\dotsm+\delta_{x_n}\right), \frac{1}{n}\left( \delta_{y_1}+\dotsm+\delta_{y_n} \right)\right)=\frac{1}{n}\min_{\tau\in S_n} \sum_{k=1}^n d(x_k, y_{\tau(k)}).
\end{equation*}
\end{proposition}
\begin{proof}
See \cite[p. 5]{villani2}.
\end{proof}
It turns out that the set \(P_\Q(X)\) is \(W^1\)-dense in \(P_1(X)\). This is the content of the following Proposition.
\begin{proposition}\label{density}
Let \((X,d)\) be a metric space and let \(\epsilon >0\) be a positive real number. If the Borel measure \(\mu\colon \mathcal{B}_X\to [0,1]\) is a Radon probability measure contained in \(P_1(X)\), then there exists a measure \(\nu_\epsilon\) contained in \(P_\Q(X)\) with \(\textrm{spt}(\nu_\epsilon)\subset \textrm{spt}(\mu)\) such that \(W^1(\mu, \nu_\epsilon)< \epsilon\).
\end{proposition}
\begin{proof}
See Theorem 6.1 in \cite{edwards} and Theorem 6.18 in \cite{villani}. 
\end{proof}
Suppose that the map \(\phi\colon X\to X\) is 1-Lipschitz. By the use of the Kantorovich-Rubinstein Duality Theorem it is readily verified that the map \(\phi_\ast\colon P_1(X)\to P_1(X)\) is well-defined and 1-Lipschitz as well.

\subsection{A 1-Lipschitz barycenter construction}
Let \((X,d)\) be a metric space. We follow K.-T. Sturm and say that a map \(\beta\colon (P_1(X), W^1)\to (X,d)\) is a \textit{contracting barycenter map} if \(\beta\) is 1-Lipschitz and if we have \(\beta(\delta_x)=x\) for all points \(x\) in \(X\), cf. \cite[Remark 6.4]{sturm}. If a metric space \((X,d)\) admits a contracting barycenter map \(\beta\colon P_1(X)\to X\), then \((X,d)\) admits a conical geodesic bicombing. Indeed, it is immediately verified that the map \(\sigma\colon X\times X\times [0,1] \to X\) given by the assignment \((x,y,t)\mapsto \beta((1-t)\delta_x + t\delta_y)\) defines a conical geodesic bicombing on \((X,d)\).

In this subsection we are interested in the reversed situation: Does every metric space that admits a conical geodesic bicombing admit a contracting barycenter map? It turns out that every complete metric space that admits a conical geodesic bicombing that has the midpoint property admits a contracting barycenter map, see Theorem \ref{navas1}.

In \cite{heinich}, A. Es-Sahib and H. Heinich developed a barycenter construction for non-empty finite subsets of separable complete Busemann spaces. Es-Sahib and Heinich's barycenter construction translates with no effort to complete metric spaces that admit a conical geodesic bicombing with the midpoint property. This is the content of the subsequent proposition.
\begin{proposition}\label{heinich}
Let \((X,d)\) be a complete metric space and let \(\sigma\colon X\times X\times [0,1]\to X\) denote a conical geodesic bicombing. If \(\sigma\) has the midpoint property, then there exists a collection \(\{b_n\colon X^n \to X\}_{n\geq 1}\) of maps that satisfies the following three conditions:
\begin{itemize}
\item(\textit{Locality}) For all integers \(n\geq 1\) and all points \(\mathbf{x}:=(x_1, \ldots, x_n)\) in \(X^n\) we have that the point \(b_n(\mathbf{x})\) is contained in \(\overline{\textrm{conv}_{\sigma}}\big(\big\{x_1, \ldots, x_n\big\}\big)\).
\item(\textit{Recursion}) For all integers \(n\geq 2\) and all points \(\mathbf{x}:=(x_1, \ldots, x_n)\) in \(X^n\) we have 
\begin{equation*}
b_n(\mathbf{x})=b_n(b_{n-1}(\mathbf{x}_1), \ldots, b_{n-1}(\mathbf{x}_n)),
\end{equation*}
where \(\mathbf{x}_k:=(x_1, \ldots, x_{k-1}, x_{k+1}, \ldots, x_n)\) for all integers \(1\leq k \leq n\).
\item(\textit{Nonexpansiveness}) For all integers \(n\geq 1\) and all points \(\mathbf{x}:=(x_1, \ldots, x_n)\) and \(\mathbf{y}:=(y_1, \ldots, y_n)\) in \(X^n\) it holds
\begin{equation*}
d(b_n(\mathbf{x}), b_n(\mathbf{y}))\leq \frac{1}{n}\min_{\tau \in S_n} \sum_{k=1}^n d(x_k, y_{\tau(k)}). 
\end{equation*}
\end{itemize}
\end{proposition}
\begin{proof}
Let \(b_1\) denote the identity map of \(X\) and define the map \(b_2\colon X^2\to X\) through the assignment \((x,y)\mapsto\sigma_{xy}(\frac{1}{2})\). It is straightforward to show that the map \(b_2\) satisfies the locality-, recursion-, and nonexpansiveness-condition. Now, we proceed by induction on \(n\geq 2\). Suppose that we have constructed a map \(b_n\colon X^n\to X\) that satisfies the locality-, recursion-, and nonexpansiveness-condition. Let \(\mathbf{x}\) be a point in \(X^{n+1}\). We define the sequence \((\mathbf{x}^{(k)})_{k\geq 1} \subset X^{n+1}\) via the recursive rule
\begin{equation*}
\mathbf{x}^{(1)}:=\mathbf{x}, \quad \mathbf{x}^{(k+1)}:=(b_n(\mathbf{x}^{(k)}_1), \ldots, b_n(\mathbf{x}^{(k)}_{n+1})),
\end{equation*}
where for each integer \(k\geq 1\) and each integer \(1 \leq l \leq n+1 \) the \(n\)-tupel \(\mathbf{x}^{(k)}_l\) is obtained from the \((n+1)\)-tupel \(\mathbf{x}^{(k)}\) by deleting the \(l\)-th entry. 
The exact same reasoning as in \cite[Proposition 1]{heinich} yields the existence of a point \(x^{\infty}\) in \(X\) such that \(\mathbf{x}^{(k)}\to (x^\infty, \ldots, x^\infty)\) with \(k\to \infty\). We set \(b_{n+1}(\mathbf{x}):=x^\infty\). It is possible to show that the map \(b_{n+1}\colon X^{n+1}\to X\) satisfies the locality-, recursion-, and nonexpansiveness-condition, cf. \cite[Proposition 1]{heinich}. 
\end{proof}
Let \((X, d)\) denote a complete metric space. Suppose that \((X,d)\) admits a conical geodesic bicombing \(\sigma\colon X\times X\times [0,1]\to X\) that has the midpoint property. Let \(\{b_n\colon X^n\to X\}_{n\geq 1}\) denote the collection of maps that we have constructed in Proposition \ref{heinich}, let \(n\geq 1\) be an integer and let \(\mathbf{x}\) be a point in \(X^n\). For every integer \(k\geq 1\) we denote by \(Q^k(\mathbf{x})\) the element in \(X^{kn}\) that is equal to \((\mathbf{x}, \ldots, \mathbf{x})\). It is tempting to assume that
\begin{equation}\label{itistempting}
b_{nk_1}(Q^{k_1}(\mathbf{x}))=b_{nk_2}(Q^{k_2}(\mathbf{x})) \quad \textrm{for all } k_1,k_2\geq 1. 
\end{equation}
However, this is not necessarily true. A counterexample can be found on page 614 in \cite{navas}.
%\begin{example}\label{exx} \normalfont{This example is due to D. Descombes. Let}  \(K_{1,3}\) denote a tripod with three leaves \(x,y,z\) and internal node \(m\), see Figure \ref{fio}.
%\begin{figure}
%\begin{center}
%\input{tripoid}
%\caption{Tripod \(K_{1,3}\) from Example \ref{exx}}
%\label{fio}
%\end{center}
%\end{figure}
%We consider \(K_{1,3}\) as a metric \(\R\)-tree and we suppose that the edge connecting \(x\) and \(m\) has length 2 and that the other two edges have length 1. Let \(d\) denote the usual tree metric. Note that the metric space \((K_{1,3}, d)\) is a complete CAT(0) space, cf. \cite[Proposition 3.4]{sturm}; hence, \(K_{1,3}\) is uniquely geodesic. Let \(\sigma\) denote the geodesic bicombing that consists of the unique geodesics of \(K_{1,3}\). Observe that the geodesic bicombing \(\sigma\) is conical and has the midpoint property. We define \(\mathbf{x}:=(x,y,z)\) and we claim that \(b_3(Q^1(\mathbf{x}))\neq b_6(Q^2(\mathbf{x}))\). An explicit calculation simplified by symmetry arguments shows that the points \(b_3(Q^1(\mathbf{x}))\) and \(b_6(Q^2(\mathbf{x}))\) lie on the edge connecting \(m\) and \(x\); but, the former has distance of \(\frac{1}{3}\) from \(m\) and the latter has distance of \(\frac{13}{45}\) from \(m\). Hence, we have shown that \(b_3(Q^1(\mathbf{x}))\neq b_6(Q^2(\mathbf{x}))\).
%\end{example}
Since the equality in \eqref{itistempting} does not hold in general, one might ask: does at least the limit 
\begin{equation}\label{navaslimit}
\lim\limits_{k\to +\infty} b_{nk}(Q^k(\mathbf{x}))
\end{equation}
exist? A. Navas showed that the limit \eqref{navaslimit} exists for all integers \(n\geq 1\) and all points \(\mathbf{x}\) in \(X^n\) if \(X\) is a complete separable Busemann space, cf. \cite[Proposition 1.2]{navas}. As Navas's proof relies solely on the fact that the collection \(\{b_n\colon X^n\to X\}_{n\geq 1}\) satisfies the recursion- and the nonexpansiveness-condition, Navas's proof translates verbatim to collections \(\{b_n\colon X^n\to X\}_{n\geq 1}\) that arose from complete metric spaces that admit a conical geodesic bicombing with the midpoint property. 

A streamlined version of Navas's proof can be found in D. Descombes's PhD thesis. If \(X\) satisfies a weak local compactness assumption, then it is possible to draw the conclusion that the limit in \eqref{navaslimit} exists via a martingale convergence theorem, cf. \cite[Theorem 2]{heinich}. Navas used the existence of the limit \eqref{navaslimit} to construct a contracting barycenter map for every complete separable Busemann space, cf. \cite{navas}. Essentially the same construction yields a contracting barycenter map for every complete metric space that admits a conical geodesic bicombing that has the midpoint property. 
\begin{theorem}\label{navas1} Let \((X,d)\) be a complete metric space and let \(\sigma\colon X\times X\times [0,1]\to X\) denote a conical geodesic bicombing. If \(\sigma\) has the midpoint property, then the map \(\beta_\sigma\colon P_\Q(X)\to X\) given by the assignment
\begin{equation}\label{baustelle}
\mu =\frac{1}{n} (\delta_{x_1}+\dotsm+\delta_{x_n}) \mapsto \lim_{k\to +\infty} b_{nk}(Q^k((x_1, \ldots, x_n)) 
\end{equation}
is well-defined and extends uniquely to a contracting barycenter map \(\beta_\sigma\colon P_1(X)\to X\) that has the following properties:
\begin{itemize}
\item(\textit{Locality}) For all measures \(\mu\) in \(P_1(X)\) we have that the point \(\beta_\sigma(\mu)\) is contained in \(\overline{\textrm{conv}_\sigma}(\textrm{spt}(\mu))\). 
\item(\textit{Equivariance}) If \(\phi\colon X\to X\) is a 1-Lipschitz and \(\sigma\)-equivariant self-map of \((X,d)\), then we have that \(\beta_\sigma\) is \(\phi\)-equivariant, that is, it holds \(\phi\circ\beta_{\sigma} = \beta_{\sigma}\circ \phi_\ast\). 
\end{itemize}
\end{theorem}
\begin{proof}
It is readily verified that the map \(\beta_\sigma\colon P_\Q(X)\to X\) is well-defined, that is, the assignment \eqref{baustelle} does not depend on the representation of \(\mu\). Let \(\mu\) and \(\nu\) denote two elements of \(P_{\Q}(X)\). Note that there is an integer \(n\geq 1\) and points \((x_1, \ldots, x_n)\) and \((y_1, \ldots, y_n)\) in \(X^n\) such that  \(\mu =\frac{1}{n} (\delta_{x_1}+\dotsm+\delta_{x_n}) \textrm{ and } \nu =\frac{1}{n} (\delta_{y_1}+\dotsm+\delta_{y_n}).\)
Due to Equation \eqref{baustelle}, the nonexpansiveness condition and Proposition \ref{formula} we have 
\begin{equation*}
d(\beta_\sigma(\mu), \beta_\sigma(\nu)) \leq \frac{1}{n} \min_{\tau\in S_n} \sum_{k=1}^nd(x_k,y_{\tau(k)})=W^1(\mu, \nu);
\end{equation*}
hence, the map \(\beta_\sigma\colon P_\Q(X)\to X\) is 1-Lipschitz. Proposition \ref{density} tells us that \(P_{\Q}(X)\) is dense in \((P_1(X), W^1)\); thus, as \(X\) is complete the map \(\beta_\sigma\colon P_\Q(X)\to X\) extends uniquely to the whole space \(P_1(X)\). We denote this map again by \(\beta_\sigma\). Note that the extended map \(\beta_\sigma\) is 1-Lipschitz by construction and we have \(\beta_\sigma(\delta_x)=x\) for all points \(x\) in \(X\); hence, the map \(\beta_\sigma\) is a contracting barycenter map on \((X,d)\). 
%Next, we show for all Radon probability measures \(\mu\) in \(P_1(X)\) that the point \(\beta_\sigma(\mu)\) is contained in \(\overline{\textrm{conv}_\sigma}(\textrm{spt}(\mu))\). Suppose that \(\mu\) is a measure in \(P_1(X)\). On account of Proposition \ref{density} there exists a sequence \(\{\nu_k\}_{k\geq 1} \subset P_\Q(X)\) with \(\textrm{spt}(\nu_k)\subset \textrm{spt}(\mu)\) and \(W^1(\mu, \nu_k)< \frac{1}{k}\) for all integers \(k\geq 1\). As a result, we have \(\overline{\textrm{conv}_\sigma}(\textrm{spt}(\nu_k))\subset \overline{\textrm{conv}_\sigma}(\textrm{spt}(\mu))\) for all integers \(k\geq 1\) and \(\beta_\sigma(\nu_k)\to \beta_\sigma(\mu)\) with \(k\to +\infty\). Now, the nonexpanisveness condition tells us that the point \(\beta_{\sigma}(\nu_k)\) is an element of \(\overline{\textrm{conv}_\sigma}(\textrm{spt}(\nu_k))\) for all integers \(k\geq 1\). Consequently, as the set \(\overline{\textrm{conv}_\sigma}(\textrm{spt}(\mu))\) is closed we obtain that the limit \(\beta_\sigma(\mu)\) is contained in \(\overline{\textrm{conv}_\sigma}(\textrm{spt}(\mu))\), as desired. 

The fact that the point \(\beta_\sigma(\mu)\) is contained in \(\overline{\textrm{conv}_\sigma}(\textrm{spt}(\mu))\) for all measures \(\mu\) in \(P_1(X)\) is a direct consequence of Proposition \ref{density}.

To conclude the proof we show that if \(\phi\colon X\to X\) is a 1-Lipschitz and \(\sigma\)-equivariant self-map of \((X,d)\), then we have that \(\beta_\sigma\) is \(\phi\)-equivariant.  As the map \(\phi\) is \(\sigma\)-equivariant, we obtain that \(\phi(b_2(x,y))=b_2(\phi(x), \phi(y))\) for all points \(x,y\) in \(X\). A straightforward induction shows for all integers \(n\geq 2\) and all points \(\mathbf{x}\) in \(X^n\) that 
\begin{equation}\label{happy}
\phi(b_n(\mathbf{x}))=b_n(\boldsymbol\phi (\mathbf{x})),
\end{equation}
where the map \(\boldsymbol\phi\colon X^n \to X^n\) is given by the assignment \((x_1,\ldots, x_n)\mapsto (\phi(x_1), \ldots, \phi(x_n))\). Suppose that \(\mu\) is a measure in \(P_\Q(X)\). There is an integer \(n\geq 1\) and a point \((x_1, \ldots, x_n)\) in \(X^n\) such that \(\mu=\frac{1}{n}(\delta_{x_1}+\dotsm+\delta_{x_n})\). Note that \(\phi_\ast\mu=\frac{1}{n}\left(\delta_{\phi(x_1)}+\dotsm+\delta_{\phi(x_n)}\right)\). We compute
\begin{equation*}
\begin{split}
\phi(\beta_\sigma(\mu))=\lim_{k\to+\infty} \phi\left(b_{nk}(Q^k(\mathbf{x}))\right)\overset{\eqref{happy}}{=}\lim_{k\to +\infty} b_{nk}(Q^k(\boldsymbol\phi(\mathbf{x})))=\beta_{\sigma}(\phi_\ast\mu).
\end{split}
\end{equation*}
Since the two \(1\)-Lipschitz maps \(\phi\circ\beta_\sigma\) and \(\beta_\sigma\circ \phi_\ast \) agree on the \(W^1\)-dense subset \(P_\Q(X)\subset P_1(X)\), we obtain that they coincide on the whole space \(P_1(X)\). The Theorem follows.
\end{proof}
We call the map \(\beta_\sigma\) from Theorem \ref{navas1} the \emph{contracting barycenter map associated to} \(\sigma\). The rest of this section is devoted to contracting barycenter maps on Banach spaces. In the subsequent proposition we show that there is precisely one contracting barycenter map on a Banach space.
\begin{proposition}\label{uniqueness} Let \((E, \norm{\cdot})\) be a Banach space, let \(\lambda\) be the conical geodesic bicombing on \(E\) that consists of the linear geodesics and let \(\beta_\lambda\colon P_1(E)\to E\) denote the contracting barycenter map associated to \(\lambda\). It holds that the map \(\beta_\lambda\colon P_1(E)\to E\) is given through the assignment
\begin{equation}
\mu \mapsto \int_{E} x \,d\mu(x)
\end{equation}
and that the map \(\beta_\lambda\) is the only contracting barycenter map on \((E, \norm{\cdot})\).
\end{proposition}
\begin{proof} Suppose that \(\beta\colon P_1(E)\to E\) is a contracting barycenter map on \((E, \norm{\cdot})\). Let \(\mu\) be a measure contained in \(P_1(E)\). The point \(\beta(\mu)\) satisfies 
\begin{equation}\label{water}
\norm{\beta(\mu)-y} \leq W^1(\mu, \delta_y)=\int_E \norm{ x-y}\,d\mu(x) \quad \textrm{ for all } y\in E.
\end{equation} 
It is well-known that \(\textrm{spt}(\mu)\) is separable and that \(\mu(E\setminus \textrm{spt}(\mu))=0\); hence, the identity map \(\textrm{id}\colon (E,\mathcal{B}_E)\to (E, \mathcal{B}_E)\) is \(\mu\)-essentially separably valued. Now, Pettis Measurability Theorem, cf. \cite[p. 124]{lang}, tells us that the identity map \(\textrm{id}\) is \(\mu\)-measurable. Hence, we can use the definition of \(P_1(E)\) and Bochner's criterion for integrability, cf. \cite[p. 142]{lang}, to deduce that the identity map \(\textrm{id}\) is Bochner integrable with respect to the measure \(\mu\). Thus, as the point \(\beta(\mu)\) satisfies the inequality \eqref{water}, Theorem 3.6 in \cite{molchanov} tells us that
\begin{equation}\label{hhighh}
\beta(\mu)=\int_E \textrm{id}(x) \,d\mu(x).
\end{equation}
Since the map  \(\beta_\lambda\) is a contracting barycenter map, we have shown that \(\beta_\lambda\) is given through the assignment \eqref{hhighh}. Furthermore, as the contracting barycenter map \(\beta\) was arbitrary, we have also shown that \(\beta_\lambda\) is the unique contracting barycenter map on \((E, \norm{\cdot})\). 
\end{proof}
Having Theorem \ref{navas1} and Proposition \ref{uniqueness} on hand we can deduce the following corollary.
\begin{corollary}\label{consequence}
Let \((E, \norm{\cdot})\) be a Banach space. If \(\mu\) is a measure in \(P_1(E)\), then the Bochner integral \(\int_E x\,d\mu(x)\) is contained in the closure of the convex hull of \(\textrm{spt}(\mu)\).
\end{corollary}
\begin {proof}
This is a consequence of Theorem \ref{navas1} and Proposition \ref{uniqueness}. 
\end{proof} 
%If the Banach space \((E, \norm{\cdot})\) in Corollary \ref{consequence} is finite dimensional, then it holds true for every measure \(\mu \) in \(P_1(E)\) that the point \(\int_E x\,d\mu(x)\) is contained in the relative interior of the convex hull of \(\textrm{spt}(\mu)\), cf. \cite[Theorem 1.48]{foellmer2004stochastic}. 
%%%%%%%%%%%%%%%%%%%%%%%%%%%%%%%%%%%%%%%%%%%%%%%%%%%%%%%%%%%%%%%%%%%%%%%%%%%%%%%%%%%%%%%%%%%%%%%%%%%%%%%%%%%%%%%%%%%%%%%%%%%%%%%%%%%%%%%%%%%%%%%%%%%%%%%%%%%%%%%%%
%%%%%%%%%%%%%%%%%%%%%%%%%%%%%%%%%%%%%%%%%%%%%%%%%%%%%%%%%%%%%%%%%%%%%%%%%%%%%%%%%%%%%%%%%%%%%%%%%%%%%%%%%%%%%%%%%%%%%%%%%%%%%%%%%%%%%%%%%%%%%%%%%%%%%%%%%%%%%%%%%
%%%%%%%%%%%%%%%%%%%%%%%%%%%%%%%%%%%%%%%%%%%%%%%%%%%%%%%%%%%%%%%%%%%%%%%%%%%%%%%%%%%%%%%%%%%%%%%%%%%%%%%%%%%%%%%%%%%%%%%%%%%%%%%%%%%%%%%%%%%%%%%%%%%%%%%%%%%%%%%%%
\section{Existence of invariant measures}\label{seven}
The primary result of this section, Theorem \ref{mainTheorem}, is proved in Subsection \ref{mainresultofthissection}. In Subsection \ref{maxval}, we introduce some results that we will need in order to derive Theorem \ref{mainTheorem}.  

\subsection{Maximal values of generalized limits}\label{maxval}
Let \(\Sigma\) denote a countable semigroup. A sequence \((\Sigma_k)_{k\geq 1}\) of non-empty finite subsets of \(\Sigma\) is a \textit{Følner sequence} if 
\begin{equation*}
\lim_{k\to +\infty} \frac{\abs{s\Sigma_k\,\triangle\,\Sigma_k}}{\abs{\Sigma_k}}=0
\end{equation*}
for all \(s\) in \(\Sigma\). Recall that the sequence \((\{0, \ldots, k-1\})_{k\geq 1}\) is a Følner sequence of the semigroup of the non-negative integers. A positive linear functional \(\Theta\colon \ell^\infty(\Sigma)\to \R\) is called \textit{generalized limit} if \(\Theta((1)_{g\in G})=1\) and \(\Theta((x_s)_{s\in \Sigma})=\Theta((x_{s_0s})_{s\in G})\) for all \(s_0\) in \(\Sigma\) and for all \(x\) in \(\ell^\infty(\Sigma)\). For convenience, we use the notation \(\mathbf{1}:=(1)_{g\in G}\) and \(s_0.x:=(x_{s_0s})_{s\in \Sigma}\) for all \(s_0\) in \(\Sigma\) and for all \(x\) in \(\ell^\infty(\Sigma)\). The subsequent Lemma is an extension of Theorem 1 in \cite{sucheston}.

\begin{lemma}\label{banachmazur} Let \(\Sigma\) be a countable semigroup and let \((\Sigma_k)_{k\geq 1}\) denote a Følner sequence of \(\Sigma\). Suppose that \(\Theta\colon \ell^\infty(\Sigma)\to \R\) is a linear functional. Then the following statements are equivalent:
\begin{enumerate}
\item The linear functional \(\Theta\colon \ell^\infty(\Sigma)\to \R\) is positive and a generalized limit.
\item The inequality 
\begin{equation*}
\Theta(x) \leq \liminf_{k\to +\infty}\left( \sup_{s\in \Sigma} \frac{1}{\abs{\Sigma_k}}\sum_{h\in \Sigma_k} x_{hs}\right)
\end{equation*}
holds for all points \(x\) in \(\ell^\infty(\Sigma)\). 
\item The inequality 
\begin{equation*}
\Theta(x) \leq \limsup_{k\to +\infty}\left( \sup_{s\in \Sigma} \frac{1}{\abs{\Sigma_k}}\sum_{h\in \Sigma_k} x_{hs}\right)
\end{equation*}
holds for all points \(x\) in \(\ell^\infty(\Sigma)\).
\end{enumerate}
\end{lemma}

\begin{proof}
We show that \( (1.)\Longrightarrow(2.)\Longrightarrow(3.)\Longrightarrow(1.)\).
\newline

\noindent
\((1.)\Longrightarrow(2.)\).  Let \(x\) in \(\ell^\infty(\Sigma)\) be a point and let \(k\geq 1\) be an integer. For every \(h\) in \(\Sigma_k\) we have that \(\Theta(h.x)=\Theta(x)\); hence, it follows that
\begin{equation*}
\Theta(x)= \frac{1}{\lvert\Sigma_k\rvert} \sum\limits_{h\in \Sigma_k} \Theta(h.x) \leq \left(\sup\limits_{s\in \Sigma} \frac{1}{\lvert\Sigma_k\rvert}\sum\limits_{h\in\Sigma_k} x_{hs}\right)\Theta(\mathbf{1})
\end{equation*}
Since \(\Theta(\mathbf{1})=1\), we have shown the desired inequality.
\newline

\noindent
\((2.)\Longrightarrow(3.)\). This is trivial. 
\newline

\noindent
\((3.)\Longrightarrow(1.)\). To begin, we show that \(\Theta\) is positive. Suppose that \(x\) in \(\ell^\infty(\Sigma)\) is a point with \(x\geq 0\). We have
\begin{equation*}
\Theta(-x) \leq \limsup_{k\to +\infty}\left( \sup_{s\in \Sigma} \frac{1}{\abs{\Sigma_k}}\sum_{h\in \Sigma_k} -x_{hs}\right)\leq 0;
\end{equation*}
hence, it follows that \(\Theta(x)\geq 0\). Next, we show that \(\Theta(\mathbf{1})=1\). Since \(\Theta(\mathbf{1}) \leq 1\) and \(\Theta(-\mathbf{1})\leq -1\), we obtain that \(\Theta(\mathbf{1})=1\), as desired. To conclude, we show that \(\Theta\) is left \(\Sigma\)-invariant. Let \(x\) in \(\ell^\infty(\Sigma)\) be a point and let \(s_0\) be an element of \(\Sigma\). We define the point \(y:=x-s_0.x\). Note that \(y\) is contained in \(\ell^\infty(\Sigma)\). We claim that \(\Theta(y)\leq 0\). We have that
\begin{equation}\label{inequalityforleftinvariance}
\begin{split}
\abs{\Theta(y)}&\leq \limsup_{k\to +\infty}\left( \sup_{s\in \Sigma} \abs{\frac{1}{\abs{\Sigma_k}}\sum_{h\in \Sigma_k}\left(x_{hs}-x_{s_0 hs}\right)}\right) \\
&\leq \norm{x}_{\infty} \limsup_{k\to +\infty}\left( \sup_{s\in \Sigma} \frac{\abs{(s_0\Sigma_k s)\Delta\Sigma_ks}}{\abs{\Sigma_k}}\right).
\end{split}
\end{equation}
Let \(s\) be an element of \(\Sigma\). Observe that since \((s_0\Sigma_k \cup \Sigma_k)s=s_0\Sigma_k s \cup \Sigma_k s\) and \((s_0\Sigma_k\cap \Sigma_k)s\subset s_0\Sigma_ks\cap \Sigma_k s\), it follows that \(s_0\Sigma_ks \Delta \Sigma_k s\subset (s_0\Sigma_k \Delta \Sigma_k)s\). As \(\abs{(s_0\Sigma_k \Delta \Sigma_k)s}\leq\abs{s_0\Sigma_k\Delta\Sigma_k}\), we obtain \(\abs{(s_0\Sigma_ks )\Delta \Sigma_k s} \leq \abs{s_0\Sigma_k \Delta \Sigma_k}\). Now, inequality \eqref{inequalityforleftinvariance} implies that \(\Theta(y)=0\), since \((\Sigma_k)_{k\geq 1}\) is a Følner sequence. Thus, we have shown that \(\Theta(s_0.x)=\Theta(x)\), as desired.
\end{proof} 
We proceed with two immediate corollaries of Lemma \ref{banachmazur}.
\begin{corollary}\label{corollaryofbanachmazur} Let \(\Sigma\) denote a countable semigroup. If \(\Sigma\) admits a Følner sequence \(\mathbold\Sigma:=(\Sigma_k)_{k\geq 1}\), then for every point \(x^\star\) in \(\ell^\infty(\Sigma)\) there is a generalized limit \(\Theta_\star^{\mathbold\Sigma}\colon \ell^\infty(\Sigma)\to \R\) such that
\begin{equation*}
\Theta_\star^{\mathbold\Sigma}(x^\star)=\limsup_{k\to +\infty}\left( \sup_{s\in \Sigma} \frac{1}{\abs{\Sigma_k}}\sum_{h\in \Sigma_k} x^\star_{hs}\right).
\end{equation*}
\end{corollary}
\begin{proof}
This is a direct consequence of Lemma \ref{banachmazur} and the Hahn-Banach Theorem.
\end{proof}
\begin{corollary}\label{corollarytwo} Let \(\Sigma\) denote a countable semigroup. If \(\mathbold\Sigma:=(\Sigma_k)_{k\geq 1}\) is a Følner sequence, then we have that the map \(\mathcal{L}_{\mathbold\Sigma}\colon \ell^\infty(\Sigma)\to \mathbb{R}\) given by the assignment
\begin{equation*}
x\mapsto \lim\limits_{k\to+\infty}\,\left(\sup\limits_{s\in \Sigma} \frac{1}{\lvert\Sigma_k \rvert} \sum\limits_{h\in \Sigma_k} x_{hs}\right)
\end{equation*}
is well-defined. Moreover, the equality \(\mathcal{L}_{\mathbold\Sigma}=\mathcal{L}_{\mathbold\Sigma^\prime}\) holds for all Følner sequences \(\mathbold\Sigma:=(\Sigma_k)_{k\geq 1}\) and \(\mathbold\Sigma^\prime:=(\Sigma_k^\prime)_{k\geq 1}\). 
\end{corollary}
\begin{proof}
This is a direct consequence of Lemma \ref{banachmazur} and Corollary \ref{corollaryofbanachmazur}.
%Let \(\mathbold\Sigma:=(\Sigma_k)_{k\geq 1}\) and \(\mathbold\Sigma^\prime:=(\Sigma_k^\prime)_{k\geq 1}\) denote two Følner sequences and let \(x^\star\) be a point of \(\ell^\infty(\Sigma)\). By the use of Corollary \ref{corollaryofbanachmazur}, we obtain two generalized limits \(\Theta_{\star}^{\mathbold\Sigma}\colon\ell^\infty(\Sigma)\to \mathbb{R}\) and \(\Theta_{\star}^{\mathbold\Sigma^\prime}\colon\ell^\infty(\Sigma)\to \mathbb{R}\) such that
%\begin{equation*}
%\Theta_{\star}^{\mathbold\Sigma}}(x^\star)=\limsup\limits_{k\to+\infty}\,\left(\sup\limits_{s\in \Sigma} \frac{1}{\lvert \Sigma_k \rvert}\sum\limits_{h\in\Sigma_k} x^\star_{hs}\right) 
%\Theta_{\star}^{\mathbold\Sigma}^\prime}(x^\star)=\limsup\limits_{k\to+\infty}\,\left(\sup\limits_{s\in \Sigma} \frac{1}{\lvert \Sigma_k^\prime \rvert}\sum\limits_{h\in\Sigma_k^\prime} x^\star_{hs}\right).
%\end{equation*}
\end{proof}
\subsection{A generalization of a Theorem of J. Oxtoby and S. Ulam}\label{mainresultofthissection}
Let \((X,d)\) be a complete separable metric space and let \(T\colon X\to X\) be a homeomorphism of \(X\). In \cite{ulamOxtoby1939}, J. Oxtoby and S. Ulam showed that if there is a point \(x_0\) in \(X\) and a compact subset \(K_0\subset X\) such that
\begin{equation*}
\limsup_{k\to +\infty}\left(\frac{1}{k} \sum_{i=1}^k \mathbbm{1}_{K_0}(T^i(x_0))\right)> 0,
\end{equation*}
then there exists a \(T\)-invariant Radon probability measure \(\mu\colon \mathcal{B}_{X}\to [0,1]\) such that \(\mu(K_0)>0\). In Oxtoby and Ulam's proof, the measure \(\mu\) is obtained by the use of Carathéodory's extension theorem from a \(T\)-invariant metric outer measure, which is constructed by the means of generalized limits. In \cite{adamski1989}, W. Adamski used the well-known construction of Radon measures via inner approximation due to Kisy\'nski and Topsøe to generalise the result of Oxtoby and Ulam to Hausdorff topological spaces. In the following we use Adamski's approach to prove a further generalisation of Oxtoby and Ulam's result. 
\begin{theorem}\label{mainTheorem}
Let \((\mathfrak{X}, \mathcal{T}_{\hspace{-0.05em}\mathfrak{X}})\) be a Hausdorff topological space and let \(\Sigma\) be a countable subsemigroup of the semigroup of continuous self-maps of \((\mathfrak{X}, \mathcal{T}_{\hspace{-0.05em}\mathfrak{X}})\). Suppose that \(\Sigma\) admits a Følner sequence \((\Sigma_k)_{k\geq 1}\). If there is a point \(x_0\) in \(E\) and a compact subset \(K_0\subset \mathfrak{X}\) such that
\begin{equation*}
\limsup_{k\to +\infty}\left(\sup_{s\in \Sigma} \frac{1}{\abs{\Sigma_k}} \sum_{h\in \Sigma_k} \mathbbm{1}_{K_0}(h\circ s(x_0))\right)> 0,
\end{equation*}
then there exists a \(\Sigma\)-invariant Radon probability measure \(\mu\colon \mathcal{B}_{\mathfrak{X}}\to [0,1]\) such that \(\mu(K_0)>0\). 
\end{theorem}
\begin{proof}
We define the sequence \(x_0:=\left(\mathbbm{1}_{K_0}(s(x_0))\right)_{s\in\Sigma}\). By the virtue of Corollary \ref{corollaryofbanachmazur} there exists a generalized limit \(\Theta\colon \ell^\infty(\Sigma)\to \R\) such that 
\begin{equation*}
\Theta(x_0)=\limsup_{k\to +\infty}\left(\sup_{s\in \Sigma} \frac{1}{\abs{\Sigma_k}} \sum_{h\in \Sigma_k} \mathbbm{1}_{K_0}(h\circ s(x_0))\right). 
\end{equation*}
The set function \(\beta\colon \mathcal{T}_{\hspace{-0.05em}\mathfrak{X}} \to [0,1]\) given by the assignment
\begin{equation*}
U\mapsto \Theta\left((\mathbbm{1}_{U}(s(x_0)))_{s\in \Sigma}\right)
\end{equation*}
satisfies \(\beta(\varnothing)=0\) and \(\beta(U\cap V)+\beta(U\cap V)=\beta(U)+\beta(V)\) for all \(U,V\) in \(\mathcal{T}_{\hspace{-0.05em}\mathfrak{X}}\). Moreover, we have for all \(U,V\) in \(\mathcal{T}_{\hspace{-0.05em}\mathfrak{X}}\) that \(\beta(U)\leq \beta(V)\), whenever \(U\subset V\). Thus, Theorem 2 in \cite{topsoe1970compactness} asserts that the map \(\mu\colon \mathcal{B}_{\hspace{-0.05em}\mathfrak{X}} \to [0,1] \) given by the assignment
\begin{equation*}
B\mapsto \sup\limits_{\substack{K\subset B,\\ K\in \mathcal{K}_{\hspace{-0.05em}\mathfrak{X}}}} \,\;\!\inf\hspace{-1.7em}\mathop{\vphantom{\sup}}\limits_{\substack{K\subset U, \\U\in\mathcal{T}_{\hspace{-0.05em}\mathfrak{X}}}} \beta(U)
\end{equation*} 
is a Radon measure. Note that \(\mu(U)\leq \beta(U)\) for all \(U\) in \(\mathcal{T}_{\hspace{-0.05em}\mathfrak{X}}\). Let \(s\) be an element of \(\Sigma\). We claim that \(s_\ast\mu=\mu\). Let \(K\) be a compact subset of \((\mathfrak{X}, \mathcal{T}_{\hspace{-0.05em}\mathfrak{X}})\). We compute
\begin{equation*}
\mu(s^{-1}(K))\leq \inf\hspace{-1.7em}\mathop{\vphantom{\sup}}\limits_{\substack{K\subset U, \\U\in\mathcal{T}_{\hspace{-0.05em}\mathfrak{X}}}} \mu(s^{-1}(U))\leq \inf\hspace{-1.7em}\mathop{\vphantom{\sup}}\limits_{\substack{K\subset U, \\U\in\mathcal{T}_{\hspace{-0.05em}\mathfrak{X}}}} \beta(s^{-1}(U))=\inf\hspace{-1.7em}\mathop{\vphantom{\sup}}\limits_{\substack{K\subset U, \\U\in\mathcal{T}_{\hspace{-0.05em}\mathfrak{X}}}} \beta(U)=\mu(K).
\end{equation*}
As a result, we obtain that \(s_\ast\mu\leq \mu\), as \(\mu\) is a Radon measure. We have
\begin{equation*}
\mu(X)=s_\ast\mu(B)+s_\ast\mu(B^c) \leq \mu(B)+\mu(B^c) =\mu(X)
\end{equation*}
for all \(B\) in \(\mathcal{B}_{\hspace{-0.05em}\mathfrak{X}}\). Hence, it follows that \(s_\ast\mu=\mu\), as claimed. By construction, we have \(\mu(K_0)>0\). Thus, by rescaling \(\mu\) if necessary we obtain a \(\Sigma\)-invariant Radon probability measure on \((\mathfrak{X}, \mathcal{T}_{\hspace{-0.05em}\mathfrak{X}})\) such that \(\mu(K_0)>0\), as desired.  
\end{proof}

\section{Proofs of the main results}\label{five}
We begin with the proof of Theorem \ref{compactfixedpoint}.
\begin{profff}
Throughout the following proof we employ the notation from Section \ref{six}. Fix a measure \(\mu_0\) in \(P(K)\). Let \(M^0(K)\) denote the vector space of all signed finite Radon measures \(\mu\colon \mathcal{B}_K\to \R\) such that \(\mu(K)=0\). The map \(\norm{\cdot}_0\colon M^0(K)\to \R\) given by the assignment
\begin{equation*}
\mu\mapsto \sup\left\{ \int_K f\,d\mu \, : \, f\colon X\to \R \textrm{ is 1-Lipschitz }\right\}
\end{equation*}
defines a norm on \(M^0(K)\), cf. \cite[Theorem 4.4]{edwards}. Due to Theorem 4.1 in \cite{edwards} we have that \(W^1(\mu, \nu)=\norm{\mu-\nu}_0\) for all measures \(\mu\) and \(\nu\) in \(P(K)\); hence, the map \(\Phi\colon P(K)\to M^0(K)\) given by \(\mu\mapsto \mu-\mu_0\) is an isometric embedding. It is well-known that the metric space \((P(K), W^1)\) is compact, cf. \cite[Remark 6.19]{villani}. As a result, the set \(\Phi(P(K))\) is a non-empty compact convex subset of \(M^0(K)\). Note that the restriction map \(s|_K\colon K\to K\) is an isometry of \(K\). For each \(s\) in \(\Sigma\) we define the map \(\overline{s}\colon \Phi(P(K))\to \Phi(P(K))\) through the assignment \(\mu-\mu_0\mapsto (s|_K)_\ast\mu-\mu_0\). Observe that \(\overline{s}\) is an affine isometry of \(\Phi(K)\). Ryll-Nardzewski's fixed-point theorem, cf. \cite{ryll1967fixed}, asserts that there is a point \(\mu_\star-\mu_0\) in \(\Phi(P(K))\) such that \(\overline{s}(\mu_\star-\mu_0)=\mu_\star-\mu_0\) for all \(s\) in \(\Sigma\). Hence, the probability measure \(\mu_\star\colon \mathcal{B}_K\to [0,1]\) is \(s|_K\)-invariant for all \(s\) in \(\Sigma\). Let \(i\colon K\hookrightarrow X\) denote the inclusion map. It is readily verified that the probability measure \(i_\ast\mu_\star\colon \mathcal{B}_X\to [0,1]\) is contained in \(P_1(X)\). Furthermore, the measure \(i_\ast\mu_\star\) is \(\Sigma\)-invariant. Let \(\beta_\sigma\colon P_1(X)\to X\) denote the contracting barycenter map associated to \(\sigma\). We define the point \(x_\star:=\beta_\sigma(i_\ast\mu_\star)\). Clearly, as \(\textrm{spt}(i_\ast\mu_\star)\) is a subset of \(K\), Theorem \ref{navas1} tells us that the point \(x_\star\) is contained in \(\overline{\textrm{conv}_{\sigma}}(K)\). Furthermore, we compute \(s(x_\star)=\beta_\sigma(s_\ast i_\ast\mu_\star)=\beta_\sigma(i_\ast\mu_\star)=x_\star\) for all \(s\) in \(\Sigma\), since \(\Sigma\) is \(\sigma\)-equivariant and \(i_\ast\mu_\star\) is \(\Sigma\)-invariant. The Theorem follows. 
\end{profff}

In order to derive Theorem \ref{maintheorem} we establish two results, Theorem \ref{maintheorem23} and Lemma \ref{bounded}, whose combination will directly imply Theorem \ref{maintheorem}. 
\begin{theorem}\label{maintheorem23} Let \((X,d)\) denote a complete metric space and let \(\sigma\colon X\times X\times [0,1]\to X\) be a conical geodesic bicombing that has the midpoint property. Suppose that \(\Sigma\) is a countable \(\sigma\)-equivariant subsemigroup of the semigroup of 1-Lipschitz self-maps of \((X,d)\). Suppose that \(\Sigma\) admits a Følner sequence \((\Sigma_k)_{k\geq 1}\). If there is a point \(x_0\) in \(X\) and a compact subset \(K_0\subset X\) such that the set \(A:=\big\{ s(x_0) : s\in \Sigma \big\}\) is bounded and the inequality
\begin{equation}\label{strictIneq2}
\limsup_{k\to +\infty}\left(\sup_{s\in \Sigma} \frac{1}{\abs{\Sigma_k}} \sum_{h\in \Sigma_k} \mathbbm{1}_{K_0}(h\circ s(x_0))\right)> 0
\end{equation}
holds, then there is a point \(x_\star\) in \(\overline{\textrm{conv}_{\sigma}}(A)\) such that \(s(x_\star)=x_\star\) for all \(s\) in \(\Sigma\). 
\end{theorem}
\begin{proof}
The intersection \(\overline{A}\cap K_0\) is a compact subset of \(X\). Theorem \ref{mainTheorem} tells us that there is a \(\Sigma\)-invariant Radon probability measure \(\mu\colon \mathcal{B}_X\to [0,1]\) such that \(\mu(\overline{A}\cap K_0)>0\). It is readily verified that the Borel measure
\begin{equation*}
\begin{split}
&\mu_\star\colon \mathcal{B}_X\to [0,1] \\
&B\mapsto\frac{1}{\mu(\overline{A})}\mu(\overline{A}\cap B)
\end{split}
\end{equation*}
is a Radon probability measure. Note that \(\overline{A}\subset s^{-1}(\overline{A})\) for all \(s\in \Sigma\). Since \(\mu\) is \(\Sigma\)-invariant, it follows that \(\mu(s^{-1}(\overline{A})\cap\overline{A}^c)=0\) for all \(s\) in \(\Sigma\). Now, it is readily verified that \(\mu_\star\) is \(\Sigma\)-invariant. By construction, the support \(\textrm{spt}(\mu_\star)\) is a subset of \(\overline{A}\). Since the subset \(\overline{A}\) is bounded, we obtain that the measure \(\mu_\star\) has a finite first moment and is thus contained in \(P_1(X)\). Let \(\beta_\sigma\colon P_1(X)\to X\) denote the contracting barycenter map associated to \(\sigma\).  We define the point \(x_\star:= \beta_\sigma(\mu_\star)\). Clearly, as \(\textrm{spt}(\mu_\star)\) is a subset of \(\overline{A}\), Theorem \ref{navas1} tells us that the point \(x_\star\) is contained in \(\overline{\textrm{conv}_{\sigma}}(\overline{A})=\overline{\textrm{conv}_{\sigma}}(A)\). Furthermore, we compute \(s(x_\star)=\beta_\sigma(s_\ast\mu_\star)=\beta_\sigma(\mu_\star)=x_\star\) for all \(s\) in \(\Sigma\), since \(\Sigma\) is \(\sigma\)-equivariant and \(\mu_\star\) is \(\Sigma\)-invariant. The Theorem follows.
\end{proof}

Note that Corollary \ref{corollarytwo} asserts that the limit \eqref{strictIneq2} does not depend on the Følner sequence \((\Sigma_k)_{k\geq 1}\). 

\begin{lemma}\label{bounded}
Let \((X,d)\) be a metric space and let \(\phi\colon X\to X\) be an isometry of \(X\). If there is a point \(x_0\) in \(X\) and a bounded subset \(B\subset X\) such that
\begin{equation*}
\limsup_{k\to +\infty}\left(\sup_{l\geq 0} \frac{1}{k} \sum_{i=0}^{k-1} \mathbbm{1}_{B}(\phi^{i+l}(x_0))\right)> 0,
\end{equation*}
then \(\phi\) has bounded orbits. 
\end{lemma}
\begin{proof}
We define the set \(A:=\big\{\phi^k(x_0) : k\geq 0\big\}\). Note that it suffices to show that \(\textrm{diam}(A)<+\infty\). We define the set \(D:=\left\{ k\geq 0 : \mathbbm{1}_{K_0}(\phi^{k}(x_0))=1 \right\}\). Theorem 3.19 (a) in \cite{furstenberg2014recurrence} asserts that there is an integer \(k_0\geq 1\) such that for every integer \(k\geq 0\) at least one of the integers \(k, k+1, \ldots, k+k_0\) is contained in the set \(D-D:=\big\{ d-d^\prime : d, d^\prime\in D, d\geq d^\prime\big\}\). We define the real number \(C:=\max\big\{ d(x_0, \phi^k(x_0)) : 0\leq k \leq k_0\big\}\). We claim that \(\textrm{diam}\left(A\right)\leq \textrm{diam}(B)+C \). Let \(k\geq 0\) be an integer. By the above there is an integer \(0 \leq l \leq k_0\) such that the integer \(k+l\) is contained in \(D-D\). We compute
\begin{equation*}
d(x_0, \phi^k(x_0))\leq d(x_0, \phi^{k+l}(x_0))+d(\phi^{k+l}(x_0), \phi^k(x_0))\leq \textrm{diam}(B)+C.
\end{equation*}
This concludes the proof, since \(\textrm{diam}\left(A\right)\leq \sup\big\{ d(x_0, \phi^k(x_0)) : k\geq 0\big\}\).
\end{proof}

\begin{proff}
Since the sequence \((\{0, \ldots, k-1\})_{k\geq 1}\) is a Følner sequence of the semigroup of the non-negative integers, it follows that Theorem \ref{maintheorem} is a direct consequence of Theorem \ref{mainTheorem} and Lemma \ref{bounded}. 
\end{proff}
%We have that \(x_\star\) is an element of \(\overline{\textrm{conv}_{\sigma}}(\overline{A})\) 

\noindent
\textbf{\textsf{Acknowledgments:}} I would like to thank Urs Lang for his careful reading of earlier draft versions of this article and his helpful comments. Furthermore, I owe thanks to Dominic Descombes for his valuable suggestions.

\renewcommand\bibsection{\section*{\refname}}
\bibliographystyle{alpha}
\bibliography{refs}
\noindent
\textsc{\small{Mathematik Departement, ETH Zürich, Rämistrasse 101, 8092 Zürich, Schweiz}}\\
\textit{E-mail adress:}{\textsf{ giuliano.basso@math.ethz.ch}}

\end{document}